\documentclass[12pt]{amsproc}
\usepackage{amsmath}
\usepackage{amssymb}
\usepackage[all]{xy}
\newtheorem{thm}{Theorem}[section]

\newtheorem{prop}[thm]{Proposition}
\newtheorem{cor}[thm]{Corollary}

\theoremstyle{definition}
\newtheorem{defn}[thm]{Definition}

\theoremstyle{remark}

\numberwithin{equation}{section}

\newcommand{\sbq}{\subseteq}

\newcommand{\mc}{\mathcal}
\newcommand{\vide}{\emptyset} 
\newcommand{\tbf}{\textbf}
\newcommand{\mbf}{\mathbf}

\newcommand{\inv}{^{-1}}
\newcommand{\bz}{\mbf{0}}
\newcommand{\C}{\text{C}}

\DeclareMathOperator{\ann}{ann}
\DeclareMathOperator{\coz}{coz}
\DeclareMathOperator{\z}{z}

\newcommand{\R}{\mathbb R}
\newcommand{\N}{\mathbb N}

\newcommand{\varep}{\varepsilon}
\newcommand{\res}{\raisebox{-.5ex}{$|$}}

\newcommand{\cx}{\text{C}(X)}

\newcommand{\In}{\text{Int}\,}
\newcommand{\qcl}{\text{Q}_{\text{cl}}}
\newcommand{\der}{\frac{d}{dx}}
\newcommand{\derk}[1]{\frac{d^#1}{dx^#1}}

\newcommand{\Hom}{\text{Hom}\,}

\newcommand{\ck}[1]{\C^{\,#1}(\R)}

\begin{document}

\hfill{Preprint version}

 \title[On $\C^k$ functions on open $U\sbq \R$]{On extending $C^{\,k}$ functions from an open set to $\R$, with applications}

 \author{W.D. Burgess}
\address{Department of Mathematics and Statistics\\ University of Ottawa, Ottawa, Canada, K1N 6N5}

\email{wburgess@uottawa.ca}
\thanks{The authors are grateful to Alan Dow for pointing them in the right direction in the case of $k=1$.  Thanks also to Simone Brugiapaglia for suggesting the Mollifier as a possible spline.}  


\author{R. Raphael}
\address{Department of Mathematics and Statistics\\ Concordia University, Montr\'eal, Canada, H4B 1R6}

\email{r.raphael@concordia.ca}

\begin{abstract} For $k\in \N\cup \{\infty\}$ and $U$ open in $ \R$, let $\C^{\,k}(U)$ be the ring of real valued functions on $U$ with the first $k$ derivatives continuous. It is shown  for $f\in \C^{\,k}(U)$ there is $g\in \ck{\infty}$ with $U\sbq \coz g$ and $h\in \ck{k}$ with $fg\res_U=h\res_U$. The function $f$ and its $k$ derivatives are not assumed to be bounded on $U$. The function $g$ is constructed using splines based on the Mollifier function. Some consequences about the ring $\ck{k}$ are deduced from this, in particular that $\qcl(\ck{k}) = \text{Q}(\ck{k})$.  \end{abstract}
\maketitle

\setcounter{section}{1} \setcounter{thm}{0}
\noindent \begin{large}\tbf{1. Introduction.} \end{large}\\[-1ex]

This note looks at certain subrings of $\C(\R)$, the ring of continuous real valued functions on the set, $\R$, of real numbers. These subrings are $\ck{k}$, the ring of $\C^k$ functions on $\R$ for $k\in \N$, i.e., the ring of functions on $\R$ with continuous derivatives up to and including the $k\,$th, as well as $\ck{\infty}= \bigcap_{k\in \N}\ck{k}$, the ring of functions having derivatives of all orders. 

The ring of continuous functions $\cx$ on a (completely regular) topological space has been very extensively studied. There is the classic text L. Gillman and M. Jerison \cite{GJ} and a multitude of papers since (see, for example, \cite{He} and \cite{AGG} and the references therein).  As a result much is known about the ring structure of the special case of $\C(\R)$, where $\R$ is the real line.  Ring theoretical properties of $\ck{k}$ have received less attention.  The impetus for this note was the paper \cite{BKR} that studied the ``domain-like objects'' in the category $\mc{S}$ of semiprime commutative rings.  The domain-like objects are the objects in the limit closure of the subcategory of domains.  This limit closure is a reflector into the subcategory. It turns out that $\C(\R)$ (in fact all $\cx$) and $\ck{\infty}$ are domain-like (\cite[Example~4.4.4 and Example~4.4.6]{BKR}) but $\ck{1}$ is not (\cite[Example~4.4.5]{BKR}), and, indeed, all $\ck{k}, k\in \N$ are not (\cite{BR}).  A description of the domain-like closures (called the DL-closures in \cite{BKR}) of $\ck{k}$ is not pursued here, rather the nature of the $\ck{k}$ as rings is studied, as is the role played by $\C^{\,k}(V)$, where $V$ is dense open in $\R$. 

The key to these results are Theorems~\ref{fg=h} and \ref{fh-k}; these are analytic statements of interest apart from the ring theoretic consequences.

Continuous functions on cozero sets and their extensions were first explored in R.L. Blair and A.W. Hager \cite[Proposition~1.1]{BH}.   They showed that, in a space $X$, if $U=\coz g$ and $f\in \C(U)$ is bounded,   then there is $h\in \cx$ such that $fg\res_U= h\res_U$.  In the present context, the real line, $\R$, is the space to be studied and more than continuity will come into play.  Consider an open subset $U\sbq \R$ and a $\C^k$ function $f$ on $U$.  If $f$ and its $k$ derivatives are all bounded on $U$ then for any $g\in \ck{k}$ where $U= \coz g$ and the $k$ derivatives of $g$ are well behaved, extending $fg\res_U$ by making it zero in $\R\setminus U$ seems feasible. However, $f$ and its $k$ derivatives need not be bounded on $U$, making the task more difficult.  Section~2 of this article is to show that such a $g$ can always be found, and that it is even possible to construct $g\in \ck{\infty}$.  This means that $fg\res_U = h\res_U$, for some $h\in \ck{k}$ and $g\res_U$ is invertible in $\C^k(U)$.  This surprising fact then has ring theoretic consequences explored in Section~3. 

Once these results have been established it can be seen that each $\C^{\,k}(U)$, $U$ open in $\R$, is a localization of $\ck{k}$ at a multiplicatively closed set of non-zero divisors. With this in hand, using $U$ dense, it will follow that the classical ring of quotients (the ring of fractions) and the complete ring of quotients of $\ck{k}$ coincide, are von Neumann regular rings and flat extensions (Theorem~\ref{Qcl}).  These are properties that $\ck{k}$ shares with $\C(\R)$ (and, indeed, all $\cx$, $X$ a metric space, N.J. Fine, L. Gillman and J. Lambek \cite[3.3]{FGL}). Moreover, all the classical rings of quotients of the $\ck{k}$, with $k\in \N\cup \{\infty\}$, are distinct (Proposition~\ref{kvsl}).

An essential tool in the analytic part is using the Mollifier, that will give a $\C^\infty$ spline, called an \emph{M-spline} in the sequel; see, for example, L. H\"{o}rmander \cite[Lemma~1.2.3]{H} but, here, used in dimension one. The construction is detailed in the next section.  \\

\noindent\tbf{Notation.}  The notation for functions in $\C(\R)$ follows that of the text \cite{GJ}. In particular, for $f\in \C(\R)$, $\coz f=\{x\in \R\mid f(x)\ne 0\}$, the cozero set, an open set, and its complement is $\z(f)$, a zero set. \\

\noindent \tbf{2. $\C^{\,k}(U)$ vs $\ck{k}$, $U$  open: using the Mollifier.} \\[-1ex]  \setcounter{section}{2} \setcounter{thm}{0}

\noindent\tbf{2a: Definition of the M-spline.} The basis for the Mollifier-spline, or \emph{M-spline} is the Mollifier function 
\begin{equation*} \sigma(x) = \begin{cases} \exp\frac{-1}{1-x^2} & \text{if $|x|<1$} \\ \;\;0&\;\text{if $|x| \ge 1$}\end{cases}\;.
\end{equation*}
Let $C= \int_{-1}^1\sigma(x)dx$.  Put $\phi(x) = (1/C) \sigma(x)$. Thus $\int_{-\infty}^\infty \phi(x)dx$ exists and is equal to 1. Now define $\Phi(x) = \int_{-\infty}^x \phi(t)dt$. It follows that $\der \Phi(x) = \phi(x)$ and, from that, all the derivatives of $\Phi(x)$ are zero at $\pm 1$, and, between $-1$ and 1, $\Phi (x)$ is a $\C^\infty$ function increasing from 0 to 1.  

If, now, the idea is to have a $\C^\infty$ spline between the points $(a,b)$ and $(c,d)$ in $\R^2$, with $a<c$, first define $\Phi_{a,c}(x) = \Phi(\frac{2x -(a+c)}{c-a})$ and then $\gamma(x) = b+(d-b)\Phi_{a,c}(x)$ is defined on the interval $[a,c]$. All the right derivatives of $\gamma(x)$ at $a$ and all the left derivatives at $b$ exist and are zero.  If, for example, $b>d$, then $\gamma(x)$ decreases as $x$ goes from $a$ to $c$.  The size of the absolute value of the $k$\,th derivative of $\Phi_{a,c}(x)$ will have to be taken into account when $b$ and $d$ are defined.

The function $\gamma(x)$ is called the \emph{M-spline} between the points $(a,b)$ and $(c,d)$. \\

\noindent\tbf{2b: The main analytic result for $k=\infty$.} 

A first theorem, and its companion for $k$ finite, pave the way for the ring theoretical results in Section~3. The proof for $k=\infty$ will be the model for that for $k$ finite. For purposes of the induction,   the 0\,th derivative of a continuous function is taken to be the function itself. 

The method used is to show that even if $f\in \C^{\,\infty}(U)$ and its derivatives are unbounded in $U$, these functions can be ``tamed'' by an appropriate choice of constants in the definition of $g$.  

\begin{thm} \label{fg=h}  
Let $\vide \ne U$ be an open set in $\R$. For $f\in \C^{\,\infty}(U)$ there is $g\in \ck{\infty}$ such that (i)~$U\sbq \coz g$, and (ii)~$fg\res_U$ extends to $h\in \ck{\infty}$. 
\end{thm}

\begin{proof} The proof will be for dense open sets because $V= U\cup \In (\R\setminus U)$ is dense in $\R$ and $f$ can be extended to $V$ by making it zero on $\In (\R\setminus U)$. From now on, it is assumed that $V$ is dense open and $f\in \C^{\,\infty}(V)$. It will be seen that in this case, $\coz g = V$.

The open set $V$ is a disjoint union of open intervals indexed over $N$, where $N=\N$ or is finite, $V=\bigcup_{m\in N}(u_m,v_m)$. The first step is to define $g$ in the interval $(u_m,v_m)$. The restriction of $g$ to this interval will be denoted $g_m$. It is first assumed that the interval is finite of length $L_m$. To simplify notation the interval will temporarily be called $(u,v)$ of length $L$. 

The interval will be subdivided with $a_1=u+L/2$ and,  for $n> 1$, $a_n = a_{n-1}+(v-a_n)/2= a_1+ L/2^n$. The idea is to connect the points $(a_n, g(a_n))$ to $(a_{n+1}, g(a_{n+1}))$   with an M-spline and with values of $g$   chosen so that the splines with their derivatives, the product of the splines with $f$ and its derivatives all go to zero as $n\to\infty$. Once this is done, a similar development can be done working from $a_1$ to the left toward $u$. The choice of the constants will take some effort.

Before continuing the definition of $g$, there needs to be a word about unbounded intervals.  It might be that there are intervals such as $(-\infty, v)$ or $(u, \infty)$ in the expression for $V$. The two cases are similar.  Assume that the $m$\,th interval is $(-\infty, v)$. Here put $u= v-1$ and $a_1= v-1/2$, and do the construction below (in Case~A) to define $g_m$ on $[a_1,v]$.  To the left of $a_1$, let $g_m$ be the constant $g_m(a_1)$. Since all the left derivatives of $g_m$ at $a_1$ are zero, $g_m$ is $\C^\infty$ on $(-\infty, a_2)$.   It will be seen that it is $\C^\infty$ on all of $(-\infty, v)$.

To define the constants, the function $f$ and its derivatives are dealt with first.  Fix $n\in\N$ and $i\ge 0$, put $$A_{n,i} =1+ \max_{x\in [a_1,a_{n+1}]} |\derk{i}f(x)|\;.$$ By continuity of the derivatives of $f$ in $V$, these constants are well-defined and, for fixed $i$, the $A_{n,i}$ are at least 1 and are non-decreasing in $n$.  

The next step is to deal with the M-splines.  This time, for $n\in \N$, and $i\ge 0$, let $$B_{n,i} = 1+\max_{j=1,\ldots,n} \left( \max_{x\in [a_j, a_{j+1}]} \{ |\derk{i} \Phi_{a_j, a_{j+1}}(x)|\}\right)\;.$$  Again, these numbers, always $\ge 1$,  are well-defined and the sequence $\{B_{n,i}\}$, for fixed $i$, is non-decreasing in $n$. In the expression $\max_{x\in [a_j, a_{j+1}]}\\| \derk{i}\Phi_{a_j, a_{j+1}}(x)|$,  the $i$\,th derivative at the endpoints should be taken to be one-sided derivatives. Both are zero.

To be able to define the function $g_m$, it will be necessary to combine these constants.  However, care must be taken to make sure that at any step in the process only a finite number of these constants are multiplied together.  There are two cases to consider; it will be seen later why there is this distinction.\\[-.5ex]

\noindent\tbf{Case A.} $L\ge 1$.  In this situation, the M-spline in the interval $[a_1, a_2]$ will begin by   only using the first derivative in the constant and the function $f$.  For each $n\in\N$, define $S_n= \prod_{i=0}^nA_{n,i} \prod_{i=0}^nB_{n,i}$.  (Notice that in the case of an infinite interval, $L$ is taken to be 1.)\\[-.5ex]

\noindent\tbf{Case B.} $L<1$.  Let $p\in \N$ be the least natural number with $L\le 1/p$. In this case, the M-spline in the interval $[a_1,a_2]$ will use all the derivatives up to the $p$\,th in the constant. For $n\le p$, define $S_n= \prod_{i=0}^p A_{n,i} \prod_{i=0}^pB_{n,i}$.  For $n> p$,   $S_n= \prod_{i=0}^nA_{n,i} \prod_{i=0}^nB_{n,i}$. \\[-.5ex]

With this groundwork, the function $g_m$ can be defined on $[a_1, v)$ using M-splines. \\[-.5ex] 

 \noindent \tbf{Definition of $\mbf{g_m}$.} Using the above data, for $x\in [a_n, a_{n+1})$ define $g_m(x) =  g_m(a_n) +(g_m(a_{n+1}) - g_m(a_n)) \Phi_{a_n, a_{n+1}}(x)$, where, for any $l\ge 1$, $g_m(a_l) = L^2/(2^{2l+1} S_l)$.  \\[-.5ex]

The function $g_m$ must also be defined on $[u, a_1]$. Here the process is similar with $(u, a_1]$ subdivided as a mirror image, say with $a_1=b_1> b_2 >\cdots$, and the constants defined in the same manner. The constants $A_{n,i}$ and $B_{n,i}$ are defined as above,  although their values may differ from those on the right half.  \\[-.5ex]

\noindent\tbf{Definition of $g$ and $h$.} Now revert to the notation for the $m$\,th interval as $(u_m, v_m)$ and $g$ can now be defined:
\begin{equation*} g(x) = \begin{cases} g_m(x) &\text{if $x\in (u_m,v_m)$, $m \in N$} \\\;\;0 & \text{if $x\notin V$}
\end{cases}\;.
\end{equation*} 

By construction $\coz g= V$. It must now be shown that $g\in \ck{\infty}$ and that $fg\res_V$, defined on $V$, extends to $h\in \ck{\infty}$. 

Both parts will be done together.  First define $h$:
\begin{equation*} h(x) = \begin{cases} f(x)g(x) &\text{if $x\in V$}\\
\;0& \text{if $x\notin V$} 
\end{cases}\;.
\end{equation*}

By construction, both $g$ and $h$ are $\C^\infty$ on $V$. The complement of $V$ has two sorts of points and they will have to be dealt with differently. 

The proof that $g$ and $h$ are in $\ck{\infty}$ will proceed by showing that they are in   $\ck{r}$, for each $r\ge 0$. This is done by induction on $r$.  

The base case, for $r =0$, only requires that $g$ and $h$ be continuous.  The continuity at points in $\R\setminus V$ must be shown.  There are two cases to be considered for $x_0\notin V$ (these cases will come up again in the induction).  \\ 

\noindent\tbf{Case~1:} $x_0$ is a left or is a right endpoint of an interval making up $V$.\\ \tbf{Case~2:} $x_0$ is not a left or is not a right endpoint of an interval making up $V$. \\[-.5ex]

In Case~1, it suffices to show that these functions go to zero at the endpoints of the intervals in $V$. Consider the interval $(u_m, v_m)$ and the right endpoint. From $a_1$ to the right, $g_m$ is a decreasing function and $g_m(a_n) = L_m^2/(2^{2n+1}S_n)$, which goes to zero as $n$ increases (since $S_n\ge 1$ and $L_m=v_m-u_m$ is fixed).   Left endpoints are dealt with in the same way.  The function $fg_m$ on the interval also goes to zero as $x$ approaches $v_m$ because for $x\in [a_n, a_{n+1})$, $|f(x)|\le A_{n,0}\le S_n$. Then, $|f(x)g(x)| \le |f(x)| \frac{L_m^2}{2^{2n+1}S_n} \le \frac{L_m^2}{2^{2n+1}}$. This goes to zero as $x\to v_m$ because $n$ will increase and $L_m$ is fixed. 

In Case~2, if $x_0\notin V$, is, say, not a right endpoint of an interval, then, by density of $V$, in any interval $(x_0-\varep, x_0)$, there are infinitely many intervals from $V$. It is here that the distinction between Cases~A and B comes up. Suppose, for $p\in \N$, that $(u_m,v_m)$ is an interval with $L_m\le 1/p$. Then $g_m(x)\le g_m(a_1) = \frac{L_m^2}{2^3}$, which goes to zero as $L_m\to 0$. Again, if $x\in [a_n, a_{n+1})$, $|f(x)g(x)| \le \frac{L_m^2}{2^{2n+1}}$. As $x\to x_0$ from the left, $x$ will either be in $\R\setminus V$ and $f(x)g(x) = 0$ or is in smaller and smaller intervals from $V$. Since, in  the calculation, $n\ge 1$, $f(x)g(x)$ will converge to zero. This shows the continuity of $g$ and of $h$, and will be the starting point of an induction. 

Throughout the proof will be for $h$ but that for $g$ will be along the same lines and is simpler, and will not be worked out in detail. \\[-.5ex]

\noindent \tbf{The induction assumption.} Assume that $r>0$ and for $0\le l <r$, (i)~$\derk{l}h$ exists and is continuous on $\R$, and (ii)~$\derk{l}h$ is zero on $\R\setminus V$, (iii)~$\derk{l}g$ exists and is continuous on $\R$, and (iv)~$\derk{l}g$ is zero on $\R\setminus V$. 

As above, it is necessary to look at points $x_0\in \R\setminus V$.  Just as before, these are of two types that need separate proofs.  \\ 

\noindent\tbf{Case 1.} Assume that the interval is $(u_m, v_m)$ and that $x_0 = v_m$. The case of a left endpoint will have a similar proof. Since $|\derk{{r-1}}h(x_0)| =0$, by the induction assumption, the expression   $$(a)\;\;\;\left| \derk{{r-1}}h(x)/(x-x_0)\right|$$ must be shown to have limit 0 as $x\to x_0$ from the left. Using the notation for the interval, it can be assumed that $x\in [a_n, a_{n+1})$ for some $n\ge 1$ and that $n\ge r$.   Note that $x_0-x\ge x_0-a_{n+1} = L/2^{n+1}$.  The expression (a) can be rewritten:$$(b)\;\;\;\left|\frac{1}{x-x_0} \sum_{i=0}^{r-1} \derk{i}f(x) \derk{{r-i-1}}g(x)\right| =$$

\begin{equation*} \begin{aligned}\frac{1}{x_0-x}\left|\sum_{i=0}^{r-1}\derk{i}f(x)\left( g_m(a_{n+1}) - g_m(a_n)\right)\derk{{r-1-i}} \Phi_{a_n,a_{n+1}}(x)\right. &\\
+\left. g_m(a_n)\derk{{r-1}}f(x)\right|\;\;\;&\\
\le \frac{2^{n+1}}{L} \left| \sum_{i=0}^{r-1}\derk{i} f(x)\left( \frac{L^2}{2^{2n+3}S_{n+1}}-\frac{L^2}{2^{2n+1}S_n}\right) \derk{{r-1-i}}\Phi_{a_n, a_{n+1}}(x)\right| &\\ +\frac{2^{n+1}}{L}\left|\frac{L^2}{2^{2n+1}S_n}\derk{{r-1}}f(x) \right|\;\;\;\end{aligned} \end{equation*}
\begin{equation*} \begin{aligned}
=\left|\sum_{i=0}^{r-1}\derk{i}f(x)\left(\frac{L}{2^{n+2}S_{n+1}}-\frac{L}{2^n S_n}\right) \derk{{r-1-i}} \Phi_{a_n, a_{n+1}}\right| &\\ + \frac{L}{2^nS_n}\left| \derk{{r-1}}f(x)\right|\;\;\;\end{aligned} \end{equation*}\vspace{-1ex}
\begin{equation*} \begin{aligned}
\le \left|\frac{L(S_n-4S_{n+1})}{2^{n+2}S_nS_{n+1}}\right| \;\; \sum_{i=0}^{r-1} \left|\derk{i}f(x)\derk{{r-1-i}}\Phi_{a_n,a_{n+1}} (x)\right| &\\
+\frac{L}{2^nS_n}\left| \derk{{r-1}}f(x)\right|\;.\;\;
\end{aligned}
\end{equation*}
However, $S_n\le S_{n+1}$ and each term of the sum is $\le S_n$; indeed,  $|\derk{i}f(x) \derk{{r-1-i}} \Phi_{a_n, a_{n+1}}(x)| \le A_{n, i}B_{n, r-1-i}\le S_n$.  Also, $|\derk{{r-1}}f(x)|\le S_n$.  Hence, the last expression above is $\le \frac{5L}{2^{n+2}S_n} rS_n + \frac{L}{2^n}= \frac{5rL}{2^{n+2}} + \frac{L}{2^n}$. Since $L$ and $r$ are fixed, as $x\to x_0$, $n$ increases and the original expression (a) goes to zero. \\[-1ex]

\noindent\tbf{Case 2.} Once again, it is here that the distinction between Cases~A and B comes up.   Here $x_0\notin V$ is not the right endpoint of an interval in $V$. This means that for $\varep>0$, by the density of $V$, the interval $(x_0-\varep, x_0)$ contains infinitely many intervals from $V$, all of length $< \varep$. For the proof it may be assumed that $\varep < 1/r$. In this case as well, it must be shown that the expression (a) tends to zero as $x\to x_0$ from the left. If $x\in (x_0-\varep, x_0)\setminus V$, then (a) is zero. From now on, assume $x\in (x_0-\varep, x_0)\cap V$. 

If $x_0-\varep$ is the right endpoint of an interval in $V$, by reducing $\varep$, it may be assumed that $x\in (u_m, v_m)$ for some interval in $V$ of length $<1/r$. Here, Case~B will apply. In this case, no matter where $x$ is in the interval, the constant $S_n$ used in the definition of $g$ will always involve the derivatives up to at least the $r-1$\,st. With this observation in hand, the proof in this case uses some of the calculations of the one for $x_0$ a right endpoint. 

The work will be done in $(u_m, v_m)$.  A factor in (a) is $1/(x_0-x)$. As before, there are the points $a_1< a_2 < \cdots <v_m$ and, on the left side, $b_1=a_1> b_2 >\cdots > u_m$.  If $x$ falls in $(u_m, v_m)$, there are two cases to consider. If $x\in [a_n, a_{n+1})$ then $x_0-x > v_m-x> v_m-a_{n+1}= L/2^{n+1}$. If $x\in [b_{n+1}, b_n)$, $x_0-x> v_m -x> x-u_m > b_{n+1} -u_m = L/2^{n+1}$. By the calculation in Case~1 in the interval $(u_m,v_m)$, (a)\,$\le   \frac{5rL}{2^{n+2}}+\frac{L}{2^n}$ in the first instance.  In the second instance with $x\in [b_{n+1},b_n)$, the calculation in Case~1 can still be used once it is noticed that $x_0-x > (1/2^{n+1})L$, as indicated above.  Then the calculation of (a) going down to $u_m$ gives (a)\,$\le \frac{5rL}{2^{n+2}}+\frac{L}{2^n}$, again. 
In both cases, $n\ge 1$ and $r$ is fixed. However, if $\varep >0$ and there is an interval $I$ from $V$ inside $(x_0-\varep, x_0)$, then the length of $I$ is $< \varep$; this implies that as $x\to x_0$, $L\to 0$, showing that $(a) \to 0$. Recall that if $x\notin V$, the expression (a) = 0, by the induction assumption.

It has thus been shown that $\derk{r}h(x)$ exists for all $x$ and has value 0 on $\R\setminus V$.  

It remains to show the continuity of $\derk{r}h$. This, in both cases, is very much like the calculations for the existence of the derivative.  Everything proceeds as before without the factor $1/(x_0-x)$ and with the factor $r+1$ instead of $r$.  

As mentioned, similar calculations, without the factor $r$ and not involving $f$, show the same results for $g$.

The induction is now complete and the functions $g$ and $h$ are in $\ck{\infty}$. 
\end{proof}

\noindent\tbf{2c. The cases $\ck{k}$, for $k\in\N$.}  If the function $f$ in the above is only in $\ck{k}$, for some $k\in\N$, then it must be shown that there is an analogue to Theorem~\ref{fg=h}. 

\begin{thm} \label{fh-k}   
Let $\vide \ne U$ be an open set in $\R$ and $k\in\N$. For $f\in \C^{\,k}(U)$ there is $g\in \ck{\infty}$ such that (i)~$U\sbq \coz g$,and (ii)~$fg\res_U$ extends to $h\in \ck{k}$. 
\end{thm} 

\begin{proof}  As in Theorem~\ref{fg=h}, $U$ may be expanded to a dense open set $V$.  The dense open set $V$ is expressed as a union of disjoint intervals, $\bigcup_{m\in N} (u_m,v_m)$.  One such finite interval is divided $a_1<a_2< \cdots <v_m$ and $a_1=b_1> b_2 > \cdots >u_m$   as above, with the same device used for any infinite intervals.  The proof follows much the same pattern as that for Theorem~\ref{fg=h} but the constants $A_{n,i}$ require modification since only the derivatives of $f$ up to the $k$\,th are available; they are given a different symbol.  In order to define $g$ and $h$, the constants $B_{n,i}$ are unchanged.   The same schema is used: for $i = 1, \ldots, k$, $\mathfrak{A}_{n,i} = 1+ \max_{x\in [a_1,a_{n+1}]} |\derk{i}f(x)|$; for $i>k$, $\mathfrak{A}_{n,i} = 1+ \max_{x\in [a_1, a_{n+1}]} |\derk{k}f(x)|$.  

The definition of the constants $S_n$, as before, divides into two cases according to the length $L$ of the interval from $V$ in question: \\
\tbf{Case A.} $L\ge 1$. (Recall that for an infinite interval, $L=1$.)   In this situation the M-spline in the interval $[a_1,a_2]$ will begin using the first derivative and the function $f$. The constant $S_n$ is given by: $$S_n= \prod_{i=0}^n\mathfrak{A}_{n,i}\prod_{i=0}^n B_{n,i}\;.$$    \\[-.5ex]

\noindent\tbf{Case B.} $L<1$.   Let $p\in \N$ be the least natural number with $L\le 1/p$. In this case the M-spline in the interval $[a_1,a_2]$ will use $A_{1,i}$ for $i=0, \ldots, p$. For $n\le p$, define $S_n= \prod_{i=0}^p \mathfrak{A}_{n,i} \prod_{i=0}^pB_{n,i}$.  For $n> p$,   $S_n= \prod_{i=0}^n\mathfrak{A}_{n,i} \prod_{i=0}^nB_{n,i}$.\\[-.5ex]

From this point on, the proof proceeds as in Theorem~\ref{fg=h} except that for $h$, the induction stops at the $k$\,th derivative, although that for $g$ can continue, making $g\in \ck{\infty}$. The induction is done in parallel for $g$ and $h$.  This needs to be made more precise as follows;  note that $k$ is at least 1. The continuity of $h$ and $g$ is as in Theorem~\ref{fg=h} giving the starting point of the induction. \\[-1ex]

\noindent \tbf{The induction assumption.}  (a)~For the induction on $h$, assume that $0< r\le k$ and for $0\le l <r$, (i)~$\derk{l}h$ exists and is continuous on $\R$, and (ii)~$\derk{l}h$ is zero on $\R\setminus V$. (b)~For the induction on $g$, assume for $0< r$ and for $0
\le l< r$ that (iii)~$\derk{l}g$ exists and is continuous on $\R$, and (iv)~$\derk{l}g$ is zero on $\R\setminus V$. 
\end{proof}

The following will be used in the next section; it is a $\C^k$ version of the fact that every open set in a metric space is a cozero set.  

\begin{cor} \label{open}  Let $k\in \N\cup \{\infty\}$. Then, for any open set $U$ in $\R$ there is $a\in \ck{k}$ such that $\coz a=U$. In fact, $a$ can be chosen to be in $\C^{\,\infty}(\R)$. \end{cor}

\begin{proof} Let $V= U\cup \In (\R\setminus U)$. Define $f\in \C^{\,k}(V)$ by $f(x) = 1$ if $x\in U$ and $f(x) =0$ otherwise. According to Theorems~\ref{fg=h} and \ref{fh-k}, there are $g, h\in \ck{k}$, $g\in \ck{\infty}$ with $fg\res_V=h\res_V$. 
Then, $h$ is zero on $(\R\setminus V) \cup \In (\R\setminus U)= \R\setminus U$. Then, choose $a=h$. 

Notice that $f\in \C^{\,\infty}(V)$, which allows $h$ to be chosen in $\ck{\infty}\sbq \ck{k}$. 
\end{proof}

The corollary also says that the set of non-zero divisors of the ring $\ck{k}$ is $\{g\in \ck{k}\mid \coz g \;\text{is dense in $\R$}\}$. Indeed, if $b\in \ck{k}$ and $\coz b$ is not dense  there is an open set $U\ne \vide$ disjoint from $\coz b$. Then if $a\in \ck{k}$ with $\coz a=U$, it follows that $ab = \bz$.  Moreover, if $U$ is chosen to be $\In \z(b)$, then $U\cup \coz b$ is dense and then $ab=\bz$ and $a+b$ is a non-zero divisor.  This property says that the ring $\ck{k}$ is \emph{complemented} (see, for example, Knox, Levy, McGovern and Shapiro, \cite[Introduction]{KLMS}). Notice that a non-zero divisor in any $\ck{k}$ is also a non-zero divisor in the larger ring $\C(\R)$. \\

\setcounter{section}{3} \setcounter{thm}{0}
\noindent \tbf{3. Applications to rings of quotients.} \\[-1.5ex]

In the situation of Theorem~\ref{fg=h}, and similarly with that of Theorem~\ref{fh-k}, there is, by restriction, a ring homomorphism $\ck{k} \to \C^{\,k}(U)$. The nature of the relationship between these two rings will be examined.

It may be necessary to recall a few terms.  The first is the notion of an \emph{epimorphism} in a category which is recalled. The category here will be that of all rings (with 1), $\mc{R}$.

\begin{defn} \label{epi}   In a category $\mc{C}$ a morphism $e\colon A\to B$ is called an 
\emph{epimorphism} if for two morphisms $f,g\colon B\to C$, $f\circ e=g\circ e$ implies $f=g$.  \end{defn}

\begin{prop} \label{zigzag}   
Let $\vide \ne U$ be an open set in $\R$. Then the homomorphism $\psi \colon \ck{k} \to \C^{\,k}(U)$, given by restriction, is a flat epimorphism in the category of rings. In particular, if $U$ is dense $\psi$ is a monomorphism.  \end{prop} 

\begin{proof} Since classical localization gives a flat epimorphism (e.g., \cite[Proposition~10.8]{S}), it suffices to show that $\C^{\,k}(U)$ is a classical localization of $\ck{k}$. Consider $M=\{a\in \ck{k}\mid U\sbq \coz a\}$.  This is a multiplicatively closed set.  If $a\in M$ then $a\res_U$ is invertible in $\C^{\,k}(U)$.  By the theorems of Section~2, if $f\in \C^{\,k}(U)$ there are $g,h \in \ck{k}$ with $U\sbq \coz g$ (i.e.,, $g\in M$), and $fg\res_U=h\res_U$, or $f= (g\res_U)\inv h\res_U$, giving the result. If $U$ is dense then the elements of $M$ are non-zero divisors, by Corollary~\ref{open}. 
\end{proof}

Notice that, as in the proposition, for a dense open $V$, the corresponding multiplicatively closed set of non-zero divisors is $M=\{g\in \ck{k}\mid V\sbq \coz g\}$.  Every non-zero divisor of $\ck{k}$ lies in one of these sets. 

See \cite{BBR} for more information about epimorphisms and the rings $\cx$: earlier works on epimorphisms in the category of rings are found in the references. 

With this in hand, it is possible to show that for $\ck{k}$, $k\in \N \cup \{\infty\}$, that the classical ring of quotients $\qcl (\ck{k})$ and the complete ring of quotients $\text{Q}(\ck{k})$ coincide and are self-injective (von Neuman) regular rings. (For the maximal flat epimorphic ring of quotients, $\text{Q}_{\text{tot}}(\ck{k})$, see \cite[XI, \S4]{S}.)

\begin{thm} \label{Qcl}  Let $k\in \N\cup \{\infty\}$.  Then, the classical ring of quotients $\qcl(\ck{k})$, the complete ring of quotients $\text{Q}(\ck{k})$, and the maximal flat epimorphism extension $\text{Q}_{\text{tot}}(\ck{k})$ coincide and is a self-injective  (von Neumann) regular ring.  \end{thm}

\begin{proof} To simplify notation, put $R=\ck{k}$ and, for dense open $V$, $R_V=\C^{\,k}(V)$. By restrictions, there is a directed diagram where for $V\sbq V'$, $\psi_{V',V}\colon R_{V'} \to R_V$, and $\psi_V\colon R \to R_V$. Call the direct limit $T$   with maps $\tau_V\colon R_V \to T$. The ring $R$ embeds in $T$.  The diagram can also be thought of as a diagram of $R$-modules, via the embeddings of $R$ in each $R_V$ and, from this, of $R$ in $T$. Hence, $T_R$ is flat because a direct limit of flat modules is flat.   

The next step is to show that $T= \qcl(R)$. An element $g\in R$ is a non-zero divisor if and only if its cozero set is dense (Corollary~\ref{open} again). If $g\in R$ with $\coz g =V$ is dense, then $g\res_V$ is invertible in $R_V$, and so its image is invertible in $L$. Moreover, if $t \in T$ with representative $f\in R_V$, by Theorem~\ref{fg=h},   there exist $g,h\in R$ with $\coz g=V$ and $fg\res_V =  h\res_V$. Then $f= h\res_V(g\res_V)\inv$, showing that $t =hg\inv \in T$. 

The next step is to show that $T= \qcl(R)$ is the complete ring of quotients. Recall that the complete ring of quotients, $\text{Q}(R)$ is the direct limit $\lim_\to \Hom_R(D,R)$, where $D$ ranges over the dense ideals of $R$ (i.e., with $\ann D= \bz$). For $D$ a dense ideal, put $V_D = \bigcup_{a\in D} \coz a$. By Corollary~\ref{open}, this is a dense open set. On the other hand, for $V$ dense open, $D_V=\{a\in R\mid \coz a\sbq V\}$ is a dense ideal because it contains non-zero divisors. In addition, for any commutative ring $S$, $\qcl(S) \sbq \text{Q}(S)$. The reasoning of \cite[page~15]{FGL} for $\C(\R)$ applies here, as well, and will only be sketched.   It will be seen that each $\Hom(D,R)$ can be embedded, as an $R$-module, in $\C^k(V_D)$. 

For $\phi\in \Hom(D,R)$, define, for $x\in V_D$, $P(\phi)(x) =\frac{\phi(d)(x)}{d(x)}$, whenever $x\in \coz d$. If $x\in \coz d\cap \coz d'$, then $\frac{\phi(d)(x)}{d(x)} = \frac{\phi(d)(x) d'(x)}{d(x)d'(x)}= \frac{\phi(dd')(x)}{dd'(x)} =\frac{\phi(d')(x) d(x)}{d(x)d'(x)} =\frac{\phi(d')(x)}{d'(x)}$. Moreover, for $d\in D$,  $P(\phi)$ is $\C^k$ on $\coz d$ and, by the above calculation, $P(\phi)$ is $\C^k$ on all of $V_D$. Then, $P\colon \Hom(D,R) \to \C^k(V_D)$ is an embedding of $R$-modules; indeed, for $r\in R$, $P(\phi r) = P(\phi)r$ because $\phi$ is an $R$-module map.  Using these embeddings, the directed diagram for $\text{Q}(R)$ embeds in that for $\qcl(R)$, giving that $\text{Q}(R) \sbq \qcl (R)$.  This with the opposite inclusion yields $\text{Q}(R) = \qcl(R)$. 

Since the complete ring of quotients of a semiprime commutative ring is always regular (e.g., Stenstr\"{o}m \cite[XII, Proposition~2.2]{S}), it follows that $\qcl(\ck{k})$ is regular. 

As final steps, it is only necessary to quote \cite[page~235, Example~2]{S}, which states that if $\qcl(R)$ is regular then $\qcl(R) = \text{Q}_{\text{tot}}(R)$, and, moreover,  this ring is self-injective by J. Lambek \cite[\S4.5, Corollary to Proposition~2]{L}. \end{proof}

For metric spaces $X$, $\qcl(\C(X)) = \text{Q}(\C(X))$ by \cite[page~20]{FGL}, in particular for $X=\R$. 

Since $\ck{k}$ is complemented (the remark after Corollary~\ref{open}) it follows that $\qcl(\ck{k})$ is regular (in fact the two properties are equivalent: D.F. Anderson and A. Badawi \cite[Theorem~2.3]{AB} and also \cite[Theorem~2.5]{KLMS}), but the theorem supplies more information than the regularity, namely that $\qcl(\ck{k})= \text{Q}(\ck{k})$.

\begin{cor}\label{Qopen}   Let $k\in \N \cup \{\infty\}$ and $\vide \ne U$ an open set in $\R$.  Then $\qcl(\C^{\,k}(U)) = Q(\C^{\,k}(U))= \text{Q}_{\text{tot}}(\C^{\,k}(U))$, is a regular ring.\end{cor}

\begin{proof} Since $U=\bigcup_{n\in \N}I_n$, a disjoint union of intervals it follows that $\C^{\,k}(U) = \prod_{n\in N}\C^{\,k}(I_n)$.  Both types of rings of quotients commute with products and, hence, it suffices to consider each $\C^{\,k}(I_n)$.  However, for each open interval $I_n$ there is a $\C^{\,\infty}$ bijection between $I_n$ and $\R$.  Theorem~\ref{Qcl} then gives the result.  
\end{proof}

Another word about the rings of quotients of $\ck{k}$: they are all distinct.

\begin{prop} \label{kvsl}  For any $k,l\in \N\cup \{\infty\}$, $k< l$, $\qcl(\ck{k}) \ne \qcl(\ck{l})$.  Moreover, for $k< l\le \infty$, $\qcl(\ck{k}) \supset \qcl(\ck{l})$. \end{prop}

\begin{proof} Consider $k\in \N$.  Let $f\in \ck{k}$ be such that, in some open set $U\ne \vide$, it does not have a continuous $k+1$\,st derivative anywhere in $U$.  (Integrating the Weierstrass continuous nowhere differentiable function sufficiently many times would give an example with $U=\R$.) If $f\in \qcl(\ck{l})$, for some $k<l\le \infty$, then there would be $g, h\in \ck{l}$ with $V=\coz g$ dense in $\R$ and $f\res_Vg\res_V= h\res_V$, or $f\res_V= (g\res_V)\inv h\res_V$.  However, this would show that $f$ has a $k+1$\,st derivative in $U\cap V$, a contradiction. 

The second statement follows because a non-zero divisor in $\ck{l}$ is also one in $\ck{k}$. \end{proof}

In fact, the situation of the proof of Proposition~\ref{kvsl} is the only case where $\qcl(\ck{k})$ can differ from $\qcl(\ck{k+1})$, i.e., where there is an element of $\ck{k}$ not having a continuous $k+1$\,st derivative anywhere on some non-empty open set. 

\begin{prop}\label{almost}
  (1)~For each $k\in \N\cup \{\infty\}$, let $L_k=\{f\in \C(\R)\mid \text{there is $V\sbq \R$, dense open, such that $f\res_V\in \C^{\,k}(V)$}\}$, then the ring $L_k\sbq \qcl(\ck{k})$.  Moreover, $\qcl(L_k) = \qcl(\ck{k})$. 

(2)~Suppose $f\in \ck{k} \setminus \ck{k+1}$ but there is a dense open set $V$ such that $f\res_V\in \C^{\,k+1}(V)$.  Then, $f\in \qcl(\ck{k+1})$.   \end{prop} 
\begin{proof} (1)~Notice that $L_k$ is a ring, $\ck{k}\subset L_k\subset \C(\R)$. Suppose $f\in L_k$, i.e., $f$ is $\C^k$ on the dense open set $V$.  There exist $g,h\in \ck{k}$ with $\coz g=V$ and $f\res_V g\res_V = h\res_V$. Hence, $f\res_V = (g\res_V)\inv h\res_V \in \C^k(V)$ and $\C^k(V)$ embeds in $\qcl(\ck{k})$. This process is independent of the choice of $V$ and clearly yields a ring homomorphism $\zeta\colon L_k\to \qcl(\ck{k})$. A function in $\C(\R)$ that is zero on a dense open set, is zero. Hence, $\zeta$ is an embedding.  However, $\ck{k}\sbq L_k$, showing that  $\qcl(L_k) = \qcl(\ck{k})$. 

(2)~The element $f$ in the statement is in $L_{k+1}$ and so part (1) gives the result. 
\end{proof}

\end{document}